\documentclass[12pt]{amsart}
\usepackage[utf8]{inputenc}

\usepackage{amsrefs}
\usepackage{supertabular}
\usepackage{amsxtra,amssymb,amsmath,amscd,url,enumitem}

\setenumerate{itemsep=5pt}
\setitemize{itemsep=5pt}

\newlist{enumth}{enumerate}{1}
\setlist[enumth]{label=\emph{(\arabic*)}, ref=(\arabic*)}

\usepackage[utf8]{inputenc}
\usepackage[english]{babel}
\usepackage{fullpage}
\usepackage{scrtime}
\usepackage[colorlinks]{hyperref}
\usepackage{datetime}
\usepackage[all]{xy}
\usepackage{tikz-cd}
\usepackage{mathrsfs}
\usepackage{xspace}

\usepackage{tensor}
\usepackage{braket}

\usepackage{todonotes}
\usepackage{tensor}

\shortdate

\DeclareMathSymbol{A}{\mathalpha}{operators}{`A}%
\DeclareMathSymbol{B}{\mathalpha}{operators}{`B}%
\DeclareMathSymbol{C}{\mathalpha}{operators}{`C}%
\DeclareMathSymbol{D}{\mathalpha}{operators}{`D}%
\DeclareMathSymbol{E}{\mathalpha}{operators}{`E}%
\DeclareMathSymbol{F}{\mathalpha}{operators}{`F}%
\DeclareMathSymbol{G}{\mathalpha}{operators}{`G}%
\DeclareMathSymbol{H}{\mathalpha}{operators}{`H}%
\DeclareMathSymbol{I}{\mathalpha}{operators}{`I}%
\DeclareMathSymbol{J}{\mathalpha}{operators}{`J}%
\DeclareMathSymbol{K}{\mathalpha}{operators}{`K}%
\DeclareMathSymbol{L}{\mathalpha}{operators}{`L}%
\DeclareMathSymbol{M}{\mathalpha}{operators}{`M}%
\DeclareMathSymbol{N}{\mathalpha}{operators}{`N}%
\DeclareMathSymbol{O}{\mathalpha}{operators}{`O}%
\DeclareMathSymbol{P}{\mathalpha}{operators}{`P}%
\DeclareMathSymbol{Q}{\mathalpha}{operators}{`Q}%
\DeclareMathSymbol{R}{\mathalpha}{operators}{`R}%
\DeclareMathSymbol{S}{\mathalpha}{operators}{`S}%
\DeclareMathSymbol{T}{\mathalpha}{operators}{`T}%
\DeclareMathSymbol{U}{\mathalpha}{operators}{`U}%
\DeclareMathSymbol{V}{\mathalpha}{operators}{`V}%
\DeclareMathSymbol{W}{\mathalpha}{operators}{`W}%
\DeclareMathSymbol{X}{\mathalpha}{operators}{`X}%
\DeclareMathSymbol{Y}{\mathalpha}{operators}{`Y}%
\DeclareMathSymbol{Z}{\mathalpha}{operators}{`Z}%

\renewcommand{\leq}{\leqslant}
\renewcommand{\geq}{\geqslant}
 
\numberwithin{equation}{section}

\newcommand{\uple}[1]{\text{\boldmath${#1}$}}

\def\setminus{\mathchoice
    {\mathbin{\vrule height .62ex width 1.61ex depth -.38ex}}
    {\mathbin{\vrule height .62ex width 1.61ex depth -.38ex}}
    {\mathbin{\vrule height .50ex width 0.85ex depth -.28ex}}
    {\mathbin{\vrule height .20ex width 0.570ex depth -.24ex}}
}

\renewcommand{\mathcal}{\mathscr}

\newcommand{\Cc}{\mathbf{C}}

\newcommand{\Zz}{\mathbf{Z}}

\newcommand{\Rr}{\mathbf{R}}

\newcommand{\proba}{\mathbf{P}}
\newcommand{\expect}{\mathbf{E}}

\def\loccit{loc.\kern3pt cit.{}\xspace}
\def\cf{see\kern.3em}
\def\Cf{See\kern.3em}
\def\eg{e.g.\kern.3em}

\def\resp{\text{resp.}\kern.3em}

\newcommand{\mods}[1]{\,(\mathrm{mod}\,{#1})}

\newcommand{\eps}{\varepsilon}
\renewcommand{\rho}{\varrho}

\DeclareMathSymbol{\gena}{\mathord}{letters}{"3C}
\DeclareMathSymbol{\genb}{\mathord}{letters}{"3E}

\theoremstyle{plain}
\newtheorem{theorem}{Theorem}[section]
\newtheorem*{theorem*}{Theorem}
\newtheorem{lemma}[theorem]{Lemma}

\newtheorem{proposition}[theorem]{Proposition}

\theoremstyle{remark}

\theoremstyle{definition}

\newtheorem{remark}[theorem]{Remark}

\renewcommand{\geq}{\geqslant}
\renewcommand{\leq}{\leqslant}

\begin{document}

\title{Rational approximation with chosen numerators}

\author{Emmanuel Kowalski}
\address[E. Kowalski]{D-MATH, ETH Z\"urich, R\"amistrasse 101, 8092 Z\"urich, Switzerland} 
\email{kowalski@math.ethz.ch}

\begin{abstract}
  We consider the problem of approaching real numbers with rational
  numbers with prime denominator and with a single numerator allowed for
  each denominator. We then present a simple application, related to
  possible correlations between trace functions and dynamical sequences.
\end{abstract}

\makeatletter
\@namedef{subjclassname@2020}{%
  \textup{2020} Mathematics Subject Classification}
\makeatother



\maketitle 

\section{Introduction}

The following statement is motivated by certain specific applications
concerning possible correlations between ``trace functions'' and
``dynamical'' sequences (see Section~\ref{sec-appli} for a concrete
statement and the notes~\cite{ergodic} for a more general perspective).

\begin{theorem}\label{th-approx}
  Let~$c>0$ be a real number with $c\leq 1/2$. There exist sequences
  $(a_p)$ indexed by prime numbers, with $a_p$ an integer such that
  $0\leq a_p<p$ for all~$p$, such that for almost all $x\in [0,1]$, the
  set of primes~$p$ with
  $$
  \Bigl|x-\frac{a_p}{p}\Bigr|\leq \frac{c}{p}
  $$
  is infinite.
\end{theorem}

So, informally, we consider a problem of diophantine approximation
where, for each denominator, only one numerator is allowed (and with the
additional restriction, coming from the original motivation, that the
denominators are primes). The approximation is of course then worse than
what is allowed by varying the numerator (and of course not every choice
of numerator can be successful).

We will give two proofs in Sections~\ref{sec-proof}
and~\ref{sec-proof2}. The first is (probably unsurprisingly) very
elementary, but has some nice aspects, especially an analogy with
sieve. The second proof is more straightforward in principle, but
involves more sophisticated ingredients, especially about the
distribution of primes.

This simple result suggests some questions:

\begin{enumerate}
\item Can one describe an explicit sequence $(a_p)$ which has the
  desired property? Note that the second proof will show that it is a
  generic property (in the sense of the natural probability measure on
  the space of sequences $(a_p)$, described below).
\item More specifically, if we define $(a_p)$ using the ``greedy''
  algorithm (taking $a_p$ for successive primes so as to always maximize
  the measure of the union of the intervals up to that point), does it
  work?
\item It is easy to see that the ``exceptional set'' will always contain
  infinitely many rational numbers.  Are there other elements in this
  exceptional set? If Yes, can we describe the elements that belong to
  it, or compute its Hausdorff dimension?
\item For suitable $(a_p)$ and $x$, can we estimate asymptotically as
  $X\to +\infty$ the number of primes $p\leq X$ such that
  $|x-a_p/p|\leq c/p$? Heuristically, one can hope to have something
  like
  $$
  2c\sum_{p\leq X}\frac{1}{p}\sim 2c\log\log X
  $$
  such primes $\leq X$; can this be established for suitable choices of
  the $a_p$? The first proof provides some weaker quantitative
  information, with high probability (with respect to~$x$).
\item What about multidimensional versions? Variants on manifolds?
\end{enumerate}

\begin{remark}
  Although problems of this kind do not seem to be standard in
  diophantine approximation, one can interpret the question roughly as
  asking whether there exists a sequence $(a_p)$ such that the set of
  points $a_p/p$ is ``eutaxic'' with respect to the radii $c/p$ (see,
  e.g., the survey of Durand~\cite[Ch.\,8]{durand}, and also Bugeaud's
  book~\cite[Ch.\,6]{bugeaud}).

  A related, but more delicate, question was solved by
  Shepp~\cite{shepp} (after previous work of Dvoretzky and others), who
  found a sharp criterion which ensures that a non-increasing sequence
  $(\ell_n)_{n\geq 1}$ of positive real numbers has the property that,
  almost surely, the union of arcs of length~$\ell_n$ with independent
  and uniform centers will cover \emph{entirely} a circle of
  length~$1$. Shepp proved that this holds if and only if
  $$
  \sum_{n\geq 1}\frac{1}{n^2}\exp(\ell_1+\cdots+\ell_n)=+\infty.
  $$

  As observed by Dvoretzky~\cite{dvoretzky}, this is a different
  condition than asking that the union covers almost all points of the
  circle, which occurs if and only if the series $\sum \ell_n$ diverges,
  by an elementary argument as in Section~\ref{sec-proof2}.

\end{remark}

\subsection*{Notation} We use the Vinogradov notation $f\ll g$ (for
complex-valued functions $f$ and $g$ defined on some set~$X$): it means
that there exists a real number $c\geq 0$ (the ``implied constant'')
such that $|f(x)|\leq cg(x)$ for all $x\in X$.

\subsection*{Acknowledgements} Thanks to Y. Bugeaud for interesting
comments and references, in particular to the work of Shepp.

\section{First proof}\label{sec-proof}

We denote by $\lambda$ the Lebesgue measure. For a sequence
$\uple{a}=(a_p)$ with $0\leq a_p<p$ for all primes~$p$, define
$$
A_{\uple{a}}=\{x\in [0,1]\,\mid\, \Bigl|x-\frac{a_p}{p}\Bigr|\leq
\frac{c}{p}\text{ for infinitely many~$p$}\}.
$$

We thus want to find $\uple{a}$ with $\lambda(A_{\uple{a}})=1$.  For
real parameters~$X$ and~$Y$ with $1\leq X<Y$, we consider the set
$$
\Omega_{X,Y}(\uple{a})=\Bigl\{ x\in [0,1]\,\mid\,
\Bigl|x-\frac{a_p}{p}\Bigr|> \frac{c}{p}\text{ for $X< p\leq Y$} \Bigr\}.
$$

We observe that the set $\mathcal{A}$ of all sequences $\uple{a}$ is
naturally a compact set, as a product of finite sets. In particular,
there is a natural product probability measure on this set, where each
$a_p$ is uniform over the integers from $0$ to $p-1$. We will use the
notation $\proba(\cdot)$ and $\expect(\cdot)$ below to indicate
probability and expectation according to this measure.

The crucial lemma is the following. We view it as a kind of sieve
statement, on average over $\mathcal{A}$.

\begin{lemma}\label{lm-sieve}
  Let
  $$
  H_{X,Y}=\sum_{X< p\leq Y}\frac{1}{p}.
  $$
  
  We have
  $$
  \expect(\lambda(\Omega_{X,Y}))
  \ll \frac{1}{H_{X,Y}},
  $$
  where the implied constant depends only on~$c$.
\end{lemma}

\begin{proof}
  Let $\varphi_p\colon [0,1]\to \{0,1\}$ denote the characteristic
  function of the interval $I_p(a_p)=[a_p/p-c/p,a_p/p+c/p]$, each being
  viewed as random variables on~$\mathcal{A}$ (the $\varphi_p$ are
  random functions, the $I_p$ are random intervals). Let
  $$
  N_{X,Y}=\sum_{X< p\leq Y}\varphi_p,
  $$
  again a random variable on~$\mathcal{A}$. We denote also
  $$
  \nu_{X,Y}=\int_0^1N_{X,Y}(x)dx
  $$
  and note that $\nu_{X,Y}=2cH_{X,Y}$, independently of the value
  of~$\uple{a}$.

  Noting that $\Omega_{X,Y}$
  is the set of those~$x$ where $N_{X,Y}=0$, we deduce from Markov's
  inequality (on $[0,1]$ with the Lebesgue measure) the upper bound
  $$
  \lambda(\Omega_{X,Y})
  \leq
  \lambda\Bigl(\Bigl\{
  x\in[0,1]\,\mid\, |N_{X,Y}(x)-\nu_{X,Y}|\geq \nu_{X,Y}
  \Bigr\}\Bigr)\leq \frac{\alpha_{X,Y}}{\nu_{X,Y}^2}
  $$
  where
  $$
  \alpha_{X,Y}=\int_0^1\Bigl(N_{X,Y}(x)-\nu_{X,Y}\Bigr)^2dx
  $$
  (again a random variable on~$\mathcal{A}$).

  We have
  $$
  \alpha_{X,Y}=\int_0^1\Bigl(\sum_{X\leq
    p\leq Y}\Bigl(\varphi_p(x)-\frac{2c}{p}\Bigr)\Bigr)^2dx=
  \sum_{X< p_1,p_2\leq Y} \int_0^1
  \Bigl(\varphi_{p_1}(x)-\frac{2c}{p}\Bigr)
  \Bigl(\varphi_{p_2}(x)-\frac{2c}{p}\Bigr)
  dx.
  $$

  For $p_1=p_2$, the integral is equal to
  $$
  \int_0^1
  \Bigl(\varphi_{p_1}(x)-\frac{2c}{p_1}\Bigr)^2dx=
  \frac{2c}{p_1}\Bigl(1-\frac{2c}{p_1}\Bigr)\leq \frac{2c}{p_1}
  $$
  (variance of a Bernoulli random variable with probability of success
  $2c/p_1$), again independently of~$\uple{a}$. Thus
  $$
  \sum_{X< p_1\leq Y}\expect\Bigl(\int_0^1
  \Bigl(\varphi_{p_1}(x)-\frac{2c}{p_1}\Bigr)^2dx\Bigr)
  \leq 2cH_{X,Y}.
  $$

  We now suppose that $p_1\not=p_2$.  We then have
  \begin{equation}\label{eq-var1}
    \int_0^1 \Bigl(\varphi_{p_1}(x)-\frac{2c}{p_1}\Bigr)
    \Bigl(\varphi_{p_2}(x)-\frac{2c}{p_2}\Bigr)dx =\lambda(I_{p_1}\cap
    I_{p_2})-\frac{4c^2}{p_1p_2},
  \end{equation}
  where the first term depends on~$\uple{a}$.

  We next estimate the expectation
  $$
  \expect\Bigl( \lambda(I_{p_1}\cap I_{p_2}) \Bigr)
  $$
  over~$\uple{a}$. For this purpose, we may (and do) assume that
  $p_1<p_2$. We then have the formula
  $$
  \expect\Bigl( \lambda(I_{p_1}\cap I_{p_2}) \Bigr)= \frac{1}{p_1p_2}
  \sum_{0\leq a<p_1}\lambda\Bigl( I_{p_1}(a)\cap \bigcup_{0\leq
    b<p_2}\Bigl[ \frac{b}{p_2}-\frac{c}{p_2},
  \frac{b}{p_2}+\frac{c}{p_2} \Bigr] \Bigr).
  $$

  Drawing a picture if need be, we get
  $$
  \expect\Bigl( \lambda(I_{p_1}\cap I_{p_2}) \Bigr)=
  \frac{1}{p_1p_2}\times p_1\times \Bigl( \frac{4c^2}{p_1}+O\Bigl(
  \frac{1}{p_2} \Bigr) \Bigr)=\frac{4c^2}{p_1p_2}
  +O\Bigl(
  \frac{1}{p_2^2}
  \Bigr).
  $$

  Combined with~(\ref{eq-var1}), this leads to
  $$
  \sum_{X\leq p_1<p_2\leq Y} \expect
  \Bigl(\int_0^1 \Bigl(\varphi_{p_1}(x)-\frac{2c}{p_1}\Bigr)
  \Bigl(\varphi_{p_2}(x)-\frac{2c}{p_2}\Bigr)dx \Bigr)\ll H_{X,Y}.
  $$

  Multiplying by two to account for the case $p_1>p_2$ and adding the
  contribution where $p_1=p_2$, we conclude that
  $\expect( \alpha_{X,Y})\ll H_{X,Y}$, and hence
  $$
  \expect(\lambda(\Omega_{X,Y}))\ll H_{X,Y}^{-1},
  $$
  as claimed.
\end{proof}

\begin{remark}
  The start of the argument is essentially of form of sieve inequality,
  especially similar to those used in certain geometric group theory
  works by Lubotzky and Meiri, see for instance the account
  in~\cite[Th.\,5.3.1]{expanders}.

  More generally, sieve methods in analytic number theory lead to bounds
  (also sometimes lower bounds) for the sizes of sets of the form
  $$
  \{n\leq N\,\mid\, n\mods{p}\notin I_p\text{ for } p\leq X\}
  $$
  for suitable choices of subsets $I_p\subset \Zz/p\Zz$ and of
  parameters~$N$ and~$X$. It was pointed out in~\cite{sieve} that in
  fact some of the basic techniques (e.g., the so-called ``large
  sieve'') can be extended to much more general settings than the
  integers, and the lemma above provides another illustration.
\end{remark}

We now conclude the proof of the theorem.  Since
$$
\lim_{Y\to +\infty}H_{X,Y}=\sum_{p>X} \frac{1}{p}=+\infty
$$
for any $X\geq 1$ (one of the most elementary quantitative forms of the
infinitude of primes, already known to Euler), Lemma~\ref{lm-sieve}
implies that for any $X\geq 2$ and any $\eps>0$, we can find
$(a_p)_{X< p\leq Y}$ (with $0\leq a_p<p$) such that
$$
  \lambda\Bigl(\Bigl\{ x\in [0,1]\,\mid\,
  \Bigl|x-\frac{a_p}{p}\Bigr|\leq \frac{c}{p}\text{ for some prime $p$
    with $X< p\leq Y$}
  \Bigr\}\Bigr)\geq 1-\eps.
$$

Let $X_1=1$. Apply the previous remark first with (say) $X=1$ and
$\eps=1/2$, and denote $X_2$ a suitable value of~$Y$. Then apply the
assumption with $X=X_2$ and $\eps=1/4$, calling~$X_3$ the value of~$Y$;
repeating, we obtain a strictly increasing sequence $(X_n)_{n\geq 1}$ of
integers and a sequence $(a_p)\in\mathcal{A}$ such that the set
$$
B_n=\Bigl\{ x\in [0,1]\,\mid\,
\Bigl|x-\frac{a_p}{p}\Bigr|\leq \frac{c}{p}\text{ for some prime $p$
  with $X_n<p\leq X_{n+1}$} \Bigr\}
$$
satisfies $\lambda(B_n)\geq 1-2^{-n}$ for any~$n\geq 1$.

If $x\in [0,1]$ belongs to infinitely sets $B_n$, then $x\in
A_{\uple{a}}$. On the other hand, since
$$
\sum_{n\geq 1}\lambda([0,1]\setminus B_n)<+\infty,
$$
the easy Borel--Cantelli Lemma shows that almost every $x\in [0,1]$
belongs at most to finitely many sets~$[0,1]\setminus B_n$.
\section{Second proof}\label{sec-proof2}

We now give the second proof. This is based on an application of
Fubini's Theorem (which is a standard approach, as in the first few
lines of Dvoretzky's paper~\cite{dvoretzky}).

We write again $I_p(\uple{a})=[a_p/p-c/p,a_p/p+c/p]$, viewed as random
intervals on the probability space~$\mathcal{A}$ to which
$\proba(\cdot)$ and~$\expect(\cdot)$ refer. Let $x\in [0,1]$. We then
have
$$
\proba(x\in I_p)=\frac{1}{p}\sum_{\substack{0\leq a<p\\|x-a/p|<c/p}}1
$$
and hence $\proba(x\in I_p)$ is either~$0$ or~$1/p$, depending on
whether there exists an integer~$a$ such that the fractional part
of~$xp$ is~$<c$, or not. 

It is a non-trivial fact from the distribution of primes that, if~$x$ is
irrational, then we have
\begin{equation}\label{eq-vino}
  \sum_{\{xp\}<c}\frac{1}{p}=+\infty
\end{equation}
(precisely, this follows by summation by parts from the much more
precise results of Vinogradov~\cite[Ch.\,XI]{vinogradov} which give an
asymptotic formula for the number of primes $p\leq X$
satisfying~$\{xp\}<c$; we note in passing that this result has been
improved since then, notably by Vaughan). Thus, since the events
$\{x\in I_p\}$ are independent by construction, the non-trivial
direction of the Borel--Cantelli Lemma implies
$$
\proba(x\in I_p\text{ for infinitely many } p)=1
$$
for any irrational~$x$.

Now by Fubini's Theorem, we obtain
\begin{align*}
\expect(\lambda(A_{\uple{a}}))&=
\expect\Bigl(
\int_{0}^1 1_{\{x\in I_p\text{ for infinitely many } p\}}\ dx
\Bigr)\\
&=\int_0^1 \proba(x\in I_p\text{ for infinitely many } p)dx
=1,
\end{align*}
and since $\lambda(A_{\uple{a}})\leq 1$, this means that $A_{\uple{a}}$
has measure~$1$ for almost all sequences~$(a_p)$. 

\begin{remark}
  We do not require the full force of Vinogradov's theorem, but in any
  case, the formula~(\ref{eq-vino}) for an arbitrary irrational
  number~$x$ seems to be comparable to the similar divergence of the sum
  of inverses of primes in an arithmetic progression.

  It would also be enough to know that the divergence of the
  series~(\ref{eq-vino}) holds for almost all~$x$ (instead of all
  irrationals), and it is quite likely that this can be proved more
  easily.
\end{remark}

\section{Application}\label{sec-appli}

Let $X=(\Rr/\Zz)^2$ and $\mu$ the Lebesgue measure on~$X$. Further, let
$f\colon X\to X$ be the map defined by $f(x,y)=(x+y,y)$. We have
$f_*\mu=\mu$. Define $\varphi\colon X\to \Cc$ by
$\varphi(x,y)=e(x)$.

We chose a sequence $(a_p)$ as in Theorem~\ref{th-approx} with
$c=1/2$. For $p$ prime and $n\in\Zz$, we define
$t_p(n)=e(-na_p/p)$. (This is a trace function modulo~$p$, but this
aspect is not important here.)

\begin{proposition}
  For $p$ prime, define $s_p\colon X\to \Cc$ by
  $$
  s_p(x,y)=\frac{1}{p}\sum_{0\leq n<p}t_p(n)\varphi(f^n(x,y)).
  $$
  
  The following properties hold:
  \par
  \emph{(1)} The sequence~$(s_p)$ does not converge almost everywhere as
  $p\to +\infty$.
  \par
  \emph{(2)} If $\mathsf{P}$ is an infinite set of primes such that
  $$
  \sum_{p\in\mathsf{P}}\frac{\log p}{p}<+\infty,
  $$
  then the sequence $(s_p)_{p\in\mathsf{P}}$ converges almost everywhere
  to~$0$.
\end{proposition}

\begin{proof}
  Since $f^n(x,y)=(x+ny,y)$ for all integers $n\in\Zz$, we can compute
  $s_n$ by summing a finite geometric progression, and we obtain
  $$
  s_p(x,y)=\frac{e(x)}{p}\frac{\sin(\pi p(y-a_p/p))}{\sin(\pi
    (y-a_p/p))}
  e\Bigl(\frac{(p-1)}{2}(y-a_p/p)\Bigr).
  $$

  It follows that $s_p(x,y)\to 0$ along any infinite set $\mathsf{P}$ of
  primes such that
  $$
  \lim_{\substack{p\to +\infty\\p\in \mathsf{P}}}
  \ p\Bigl|y-\frac{a_p}{p}\Bigr|=+\infty.
  $$

  Thus~(2) follows because the assumption there implies that almost all
  $(x,y)$ satisfy
  $$
  \Bigl|y-\frac{a_p}{p}\Bigr|\geq \frac{\log p}{p}
  $$
  for all but finitely many $p\in\mathsf{P}$, by the easy
  Borel--Cantelli lemma.

  We now prove~(1). Note that if $(s_p)$ converges almost everywhere,
  the limit must be zero according to~(2). But the formula for $s_p$ and
  the defining property of $(a_p)$ imply that for all~$x$ and almost all
  $y\in\Rr/\Zz$, we have $|s_p(x,y)|\gg 1$ for infinitely many
  primes~$p$. Thus there is almost surely a subsequence which does not
  converge to~$0$.
\end{proof}

\begin{remark}
  The condition in~(2) can be replaced by
  $$
  \sum_{p\in\mathsf{P}}\frac{\psi_p}{p}<+\infty,
  $$
  where $(\psi_p)_p$ is an arbitrary sequence of non-negative real
  numbers such that $\psi_p\to +\infty$.
\end{remark}


\begin{thebibliography}{CCC}

\bibitem{bugeaud} Y. Bugeaud: \textit{Approximation by algebraic
    numbers}, Cambridge Tracts in Math. 160, Cambridge Univ. Press,
  2004.

\bibitem{durand} A. Durand: \textit{Describability via ubiquity and
    eutaxy in Diophantine approximation}, Ann. Math. Blaise Pascal 22
  (2015), 1--149.  

\bibitem{dvoretzky} A. Dvoretzky: \textit{On covering a circle by
    randomly placed arcs}, Proc. Nat. Acad. Sci. 42 (1956), 199--203.
  
\bibitem{ik} H. Iwaniec and E. Kowalski: \textit{Analytic Number
    Theory}, A.M.S Colloquium Publ. 53, 2004.

\bibitem{sieve} E. Kowalski: \textit{The large sieve and its
    applications}, Cambridge Tracts in Math. 175, Cambridge Univ. Press,
  2008.
  
\bibitem{expanders} E. Kowalski: \textit{An introduction to expander
    graphs}, Cours Spécialisés 26, S.M.F, 2019.
  
\bibitem{ergodic} E. Kowalski: \textit{Unmotivated ergodic averages},
  preprint (2019--2023); available at
  \url{www.math.ethz.ch/~kowalski/ergodic-trace.pdf}.

\bibitem{shepp} L.A. Shepp: \textit{Covering the circle with random
    arcs}, Israel J. Math. 11 (1972), 328--345.

\bibitem{vinogradov} I.M. Vinogradov: \textit{The method of
    trigonometrical sums in the theory of numbers}, Interscience
  Publishers, 1963.
  
\end{thebibliography}
\end{document}